\documentclass[12pt,reqno]{article}

\usepackage{amsthm,amsfonts,amssymb,amsmath,epsf, verbatim}

\newtheorem{theorem}{Theorem}
\newtheorem{lemma}[theorem]{Lemma}
\newtheorem{corollary}[theorem]{Corollary}
\newtheorem{proposition}[theorem]{Proposition}

\newtheorem{conjecture}[theorem]{Conjecture}

\begin{document}

\title{On the Third Largest Prime Divisor of an Odd Perfect Number}

\author{}
\date{}

\maketitle
\vspace{-1 cm}

\begin{center}{\bf Sean Bibby}\\
sean.bibby@mail.mcgill.ca \\ 
{\bf Pieter Vyncke }\\
pieter.vyncke@hotmail.com \\ 
{\bf Joshua Zelinsky} \\
{\it Department of Mathematics, Iowa State University}\\  zelinsky@gmail.com
\end{center}

\begin{abstract} Let $N$ be an odd perfect number and let $a$ be its third largest prime divisor, $b$ be the second largest prime divisor, and $c$ be its largest prime divisor. We discuss steps towards obtaining a non-trivial upper bound on $a$, as well as the closely related problem of improving bounds for $bc$ and $abc$. In particular, we prove two results. First, we prove a new general bound on any prime divisor of an odd perfect number and obtain as a corollary of that  bound that $$a < 2N^{\frac{1}{6}}.$$ Second, we show that $$abc < (2N)^{\frac{3}{5}}.$$ We also show how in certain circumstances these bounds and related inequalities can be tightened.  

Define a $\sigma_{m,n}$ pair to be a pair of primes $p$ and $q$ where $q|\sigma(p^m)$ and $p|\sigma(q^n)$. Many of our results revolve around understanding $\sigma_{2,2}$ pairs. We also prove  results concerning $\sigma_{m,n}$ pairs for other values of $m$ and $n$. 

\end{abstract}

\section{Introduction}

Let $N$ be an odd perfect number. Assume that $N={p_1}^{a_1}{p_2}^{a_2}\cdots {p_k}^{a_k}$ where $p_1, p_2, \cdots, p_k$ are primes satisfying $p_1 < p_2 < p_3 <  \cdots < p_k$. Acquaah and Konyagin \cite{AK} proved that one must have \begin{equation}\label{AK inequality} p_k < (3N)^{1/3}.\end{equation} The third  author \cite{Zelinskysecondlargest} proved that \begin{equation}\label{b bound from Z1} p_{k-1} < (2N)^{1/5}. \end{equation} In this article we prove that $p_{k-2} < (2N)^{1/6}$ and discuss possible directions for further improvement. Iannucci \cite{Iannuccithird} proved a lower bound of $p_{k-2}> 100$. \par  In \cite{Zelinskysecondlargest}, the third  author also proved that \begin{equation}\label{bc inequality from previous paper}  p_kp_{k-1} < 6^{1/4}N^{1/2}. \end{equation}

Using closely related techniques, Luca and Pomerance    \cite{LucaPomerance} proved that $$p_1p_2p_3 \cdots p_k < 2N^{\frac{17}{26}}.$$ That result was subsequently improved by Klurman \cite{Klurman} who replaced the exponent of $\frac{17}{26}$ with $\frac{9}{14}$. Klurman's improvement of the exponent came at the cost of replacing the 2 in front with a non-explicit constant.  A long-term goal of many researchers has been to try to show that one in fact has \begin{equation}p_1p_2 \cdots p_k < N^{\frac{1}{2}} \label{radicalineq}. \end{equation} A large amount of computation has been expended on showing that an odd perfect number which violates Inequality (\ref{radicalineq}) must be very large and have very large prime factors (see \cite{Ellia}, \cite{OchemRaoRadical}).

Euler proved the following result which is often the starting point for any work on odd perfect numbers. 

\begin{lemma}\label{Euler form for OPN} If $N$ is an odd perfect number then we have $N= p^em^2$ for some prime $p$ where $(p,m)=1$ and $p \equiv e \equiv 1$ (mod 4).

\end{lemma} We will refer to the prime raised to an odd power in the factorization of $N$ as the ``special prime''. It follows immediately from Euler's result that one must have $p_{k-2} < N^{1/5}$. More generally, it follows immediately from Euler's theorem that for all $i$ with $1 \leq i \leq k $, $$p_{k-i} < (2N)^{\frac{1}{2i+1}}.$$  It is worth realizing how weak a result Euler's result is; Euler's result applies not just to odd perfect numbers, but  to any odd number $n$ where $\sigma(n) \equiv 2$ (mod 4). 

We will for the remainder of this paper, when convenient, use a slightly different notation for an odd perfect number which will allow us to avoid the frequent use of subscripts. In particular, we will also write $a=p_{k-2}$, $b=p_{k-1}$, and $c=p_k$. For a prime $p$ and integers $n$ and $s$, we will write $p^s||n$ to mean that $p^s|n$ and that $p^{s+1}\not|n$. When this is the case we will refer to $p^s$ as a component of $n$. 

We first note that we have the following upper bound on any prime factor. 

\begin{theorem}
\label{Generalbound} Let $N$ be an odd perfect number. We have for any integer $i$ with $0 \leq i \leq k-1$, 
$$p_{k-i} < (2N)^{\frac{1}{2i+2}}.$$
\end{theorem}

\begin{proof} We note that for $i=0$, this result is a corollary of Acquaah and Konyagin's bound. For $i=k-1$, the result follows from the well known fact that an odd perfect number must be divisible by a fourth power of a prime. Suppose that $1 \leq i \leq k-2$. Consider $$M=\prod_{{k-i} \leq j \leq k} p_j^{a_j}.$$ Note that $M$ must be deficient since it is a proper divisor of a perfect number. Thus, one must have $M<\sigma(M) <2M$. Thus, there exists $j$ such that  $j \geq k-i$ and satisfying $p_j^{a_j}\not|\sigma(M)$. Since $N$ is perfect, but any proper divisor is deficient, there is some $\ell < k-i$ such that  $p_j|\sigma(p_\ell^{a_\ell})$. Hence, $p_\ell^{a_\ell} > \frac{1}{2} p_{k-i}$. We then have
$$ \left(\frac{1}{2}p_{k-i}\right)p_{k-i}^{a_{k-i}}p_{k-i+1}^{a_{k-i+1}} \cdots p_k^{a_k} <  p_\ell^{a_\ell}M \leq  N.$$ Lemma \ref{Euler form for OPN} implies that  at most one of our exponents $a_m$ can be 1, and thus we have 
$$(1/2)p_{k-i}^{2i+2}  <   \left(\frac{1}{2}p_{k-i}\right)p_{k-i}^{a_{k-i}}p_{k-i+1}^{a_{k-i+1}} \cdots p_k^{a_k}.$$ From the above inequalities we then have that $$(1/2)p_{k-i}^{2i+2} < N,$$ and hence
$p_{k-i} < (2N)^{\frac{1}{2i+2}}.$

\end{proof}

We will make frequent use of the  argument used here where $N$ being perfect will force the existence of an additional component to supply a prime to $\sigma(N)$. We will refer to this as an $m$-type argument. 

We then obtain an immediate consequence of Theorem \ref{Generalbound}.

\begin{corollary} We have $a < 2N^{\frac{1}{6}}.$
\end{corollary}

Many of the prior results on upper bounding the larger prime factors of an odd perfect number can be thought of as statements that involve restrictions a $\sigma_{m,n}$ pair can look like. By a $\sigma_{m,n}$ pair we mean a pair of primes $p$ and $q$ where $q|\sigma(p^m)$, and $p|\sigma(q^n)$.  \\

Consider the following Lemma from \cite{Zelinskysecondlargest}.

\begin{lemma}\label{p|q+1 and q|p^2+p+1 lemma} If $p$ and $q$ are positive odd integers such that $q|p^2+p+1$ and $p|q+1$, then we must have $(p,q) = (1,1)$ or $(p,q)=(1,3)$. 
\end{lemma}

Lemma \ref{p|q+1 and q|p^2+p+1 lemma}  leads to the result that there are no $\sigma_{1,2}$ pairs. 
Note that a $\sigma_{m,n}$ pair has a graph-theoretic interpretation: Given an odd perfect number $p_1^{a_1}p_2^{a_2} \cdots p_k^{a_k}$, we can construct a directed graph where, for every $i$ with $1 \leq i \leq k$, each vertex is labeled with $p_i$. For vertices with labels $p_i$ and $p_j$, there is an arrow from a vertex $p_i$ to a vertex $p_j$ if $p_i|\sigma(p_j^{a_j}).$ We can give a weight  $m$ to each directed edge, where $$p_i^m||\sigma(p_j^{a_j}).$$  A $\sigma_{m,n}$ pair corresponds to a 2-cycle in this graph.  Note that other results about odd perfect numbers can be thought of as statements about this graph; for example, see Theorem 2 of \cite{Dandapat}.

One of the primary obstructions to proving strong results  is the possibility of the presence $\sigma_{2,2}$ pairs. That is, primes $p$ and $q$ where $p|q^2+q+1$ and $q|p^2+p+1$.  Examples are $(3,13)$ and $(13,61)$. If these were the only $\sigma_{2,2}$ pairs, much of what we do here would be simplified. Unfortunately, there's at least one very large solution: $$(p,q)= (22419767768701,107419560853453).$$ 

We will define a {\emph{quasisolution}} to be a pair of positive integers $p$ and $q$ where $p|q^2+q+1$ and $q|p^2+p+1$. Notice that we do not require the $p$ and $q$ in a quasisolution to be prime. One major step in understanding $\sigma_{2,2}$ pairs is to completely classify quasisolutions.

\begin{lemma}\label{quasisolutions only come from quasisolution equation}   Let $p$ and $q$ be positive integers. Then $p,q$ form  quasisolution if, and only if, they satisfy \begin{equation}\label{quasisolution equation}5pq=p^2+q^2+p+q+1.\end{equation} Every quasisolution is given by a consecutive pair of terms in the sequence given by $t_1=t_2=1$ and with $$t_{n+2}=\frac{t_{n+1}^2 + t_{n+1} +1}{t_n}. $$
Finally, we have 
\begin{equation}\label{t_n approximate inequality} 4t_n< t_{n+1} < 5t_n.
\end{equation} for all $n>3$.\footnote{ Versions of  Lemma \ref{quasisolutions only come from quasisolution equation} have been proven in other locations also. See, for example \cite{Mills}, which proves a more general result. Interest in $\sigma_{2,2}$ pairs has also arisen in at least one other completely different context. See \cite{CaiShenJia}.}
\end{lemma}
\begin{proof}  It is immediate that if $p$ and $q$ satisfy Equation \ref{quasisolution equation}, then  $p|q^2+q+1$ and $q|p^2+p+1$ and hence they are a quasisolution. 
If $(p,q)$ is a quasisolution with $p < q$, and $d= \frac{q^2+q+1}{p}$, then a little algebra shows that $(q,d)$ is a quasisolution  with $q < d$.  Thus, given a quasisolution, we can repeatedly apply this process to get a chain of quasisolutions which we will call a quasichain. For any such quasichain, we have $ \frac{p^2+q^2+p+q+1}{pq} = \frac{q^2+d^2+q+d+1}{qd}.$ Thus, for any quasisolution, we may look at the quantity $$m(p,q) = \frac{p^2+q^2+p+q+1}{pq}$$ which is an invariant for the entire quasichain. So, if we can prove that every quasisolution arises from the quasichain which starts off with $p=1,q=1$ then we are done.

Let $x_n$ be a chain of quasisolutions. Note that  by rearranging our definition of how to extend a quasichain we have that \begin{equation}\label{descending quasisolution equation} x_{n}= \frac{x_{n+1}^2+x_{n+1}+1}{x_{n+2}}.\end{equation} Note also that every member of a quasichain must be odd (because for any integer $t$, $t^2+t+1$ is odd). If $x_{n+1}$ and $x_{n+2}$ are both greater than 1, then it is easy to check that $x_n$ is positive and satisfies $x_n < x_{n+1}+2$. Since  both $x_n$ and $x_n+1$ are odd, one must have $x_n \leq x_{n+1}$, and it is easy to see that equality can occur only when $x_n=x_{n+1}=1$. Thus, by the well-ordering principle for any chain we can keep taking smaller and smaller elements until we reach a lowest term. This term must be of the form $(1,x)$ for some $x$. Such a term must satisfy $x|1^2+1+1=3$. So the only possible options for $x$ are $x=1$ and $x=3$. Since these are the first two terms of the chain which starts with $(1,1)$, we have proven the first part of the Lemma. 

Once we have that all quasisolutions arise this way, Inequality (\ref{t_n approximate inequality}) arises from a straightforward induction argument.

\end{proof}

For the remainder we will write $t_n$ to denote the sequence formed by the chain of quasisolutions. That is, $t_1=1$, $t_2 =1$, and in general $$t_{n+2} = \frac{t_{n+1}^2 +t_{n+1} +1}{t_n}.$$ 

We will use this characterization of quasisolutions to substantially restrict what $\sigma_{2,2}$ pairs can look like. Before we do so, we note that the characterization of quasisolutions allows one to easily search for $\sigma_{2,2}$ pairs. A computer search shows that after the large pair mentioned above, there are no $\sigma_{2,2}$ pairs below $10^{4000}$. \\

Let $w$ be a positive integer where $w$ has no prime divisors which are $1$ (mod 3). One can easily see that the sequence $x_n$ (mod $w$) is periodic. Moreover, $x_n$ mod $w$ will always have a symmetry to it: We will not need this general symmetry, but it is worth noting and is well illustrated by $w=11$. We have (mod 11) the sequence $$1,\:1,\:3,\:2,\:6,\:5,\:7,\:7,\:5,\:6,2,\:3,\:1,\:1 \cdots.$$ Notice that after we reach the pair of 7s, the sequence repeats itself in reverse order until reaching $1,1$, where the pattern will then restart. This is due to the symmetry in the definition of our recursion. In particular, that $$t_nt_{n+2} = t_{n+1}^2 + t_{n+1} +1.$$ We also note that we have the following other behavior:  $t_n \equiv 1$ (mod 4) except if $n \equiv 0$ (mod 3). Similarly, $t_n \equiv 1$ (mod 3) except when $n \equiv 0$ (mod 3), in which case $t_n \equiv 0$ (mod 3). Thus, we immediately have that any $\sigma_{2,2}$ pair must have $p \equiv q \equiv 1$ (mod 4). 

We note that mod $5$, the sequence has period 4 with $t_n \equiv 1$ when $n \equiv 1$ or $2$ (mod 4), and $t_n \equiv 3$ when $n \equiv 3$ or $0$ (mod 4).


\begin{lemma}\label{squared divisor in sigma 2,2 pair} There are no primes $p$ and $q$ with $p^2|q^2+q+1$ and $q|p^2+p+1$.
\end{lemma} 
\begin{proof}
Assume we have such a pair. Note that $p$ and $q$ must be both $1$ (mod 3) (since any divisor of $n^2+n+1$ is 1 or 0 mod 3). So we have $3p^2|q^2+q+1$, and $3q|p^2+p+1$. We will choose $k$ such that $kp^2 = q^2+q+1$. First consider the possibility that $k=3$, that is, $3p^2= q^2 + q+1$. 


Then  $$5pq =q^2 + q + 1  + p^2 +p = 3p^2 + p^2 + p = 4p^2 + p.$$
Thus, we have $$5q = 4p+1.$$ Since $q|4p+1$ and $q|p^2+p+1$, we have that $$q|4(p^2+p+1) - p(4p+1)=3p+4. $$ Since $q|3p+4$ and $q|4p+1$, we must have $$q|4p+1 - (3p+4) = p -3$$ which is impossible since $q > p$. 

Thus, we may assume that $kp^2 = q^2 + q +1$ for some $k>3$. Note that $k \equiv 3$ (mod 6). Note also that $k \not \equiv 0$ (mod 9) since $n^2 + n +1 \equiv 0$ (mod 9) has no solutions. Also, $k$ cannot be divisible by $5$, since $5\equiv 2$ (mod 3). Thus, we have that $k \geq 21$. 

Let us assume that $k=21$. We then have that $21p^2 = q^2 + q +1$, and using similar logic as before, we have that $$5pq = 21p^2 + p^2 + p= 22p^2 +p.$$
We thus have $$5q = 22p+ 1.$$ We then obtain a contradiction very similarly to how we obtain a contradiction  for $k=3$. Since $q|22p+1$ and $q|p^2+p+1$, we have that $$q|(22p^2 + 22p + 22) - p(22p+1) = 21p + 22.$$ Thus, $q|(22p+1) - (21p + 22) = -21$, and we can check that neither $q=3$ nor $q=7$ works. 

Thus, we have that $k \neq 21$. The next acceptable value for $k$ is $k = 33$ (we cannot have $k=27$ since $9|27$). So, $k \geq 33$. We then have that $$33p^2 \leq q^2 + q +1 $$ which implies that $q > 5p$ and hence contradicts Inequality (\ref{t_n approximate inequality}). 
\end{proof}

Lemma \ref{squared divisor in sigma 2,2 pair} has a graph theoretic interpretation in that the graph of an odd perfect number cannot have a pair of vertices $x$ and $y$, each with out-degree 2, with vertex $x$ pointing to vertex $y$ and with $y$ pointing only to vertex $x$. 

Lemma \ref{squared divisor in sigma 2,2 pair} also naturally leads to the next Lemma.

\begin{lemma}\label{PQR lemma} There are no primes $p,q,r$ with $pr|q^2+q+1$, $q|p^2+p+1$, $p|r+1$, and $r \equiv 1$ (mod 4). 
\end{lemma}
\begin{proof} Assume we have three such primes.  Note that the first and third division relations imply that we must have $p < q$. We may also, by a straightforward computation, assume that $q >p > 21$. 

We have $yp=r+1$ for some $y \equiv 0$ (mod 2). We have $prx= q^2+q+1$ for some $x$ with $x \equiv 0$ (mod 3), and we have $p \equiv q \equiv r \equiv 1$ (mod 3).  Note that we cannot have $x \equiv $ 0 (mod 9), and we cannot have $5|x$, since there are no solutions to $q^2 + q + 1 \equiv 0$ (mod 5). Thus, if $x \neq 3$, we must have $x \geq 21$. But if we are in this situation we can use that  $p>21$ to obtain $$41p^2 < 21p(2p-1) \leq q^2+q+1,$$ implies that $5p < q$. But that contradicts Inequality (\ref{t_n approximate inequality}), since $p$ and $q$ are a quasisolution. Thus, we must have $x=3$. Similarly, we must have $y \equiv 2$ (mod 4). So, if $y>2$, then one must have either $y=6$ or $y \geq 10$. $y=6$ leads to a contradiction since $q^2+q+1$ would then have a $2$ (mod 3) divisor, so one would need to have $y \geq 10$. Since $p >21$  one has  $$29p^2 < 3p(10p-1) \leq q^2 + q+1,$$ which implies that $5p < q$ which again leads to a contradiction with Inequality (\ref{t_n approximate inequality}). Thus, we must have $x=3$ and $y=2$. We then have $3p(2p-1)=q^2+q+1$ which implies that $4p < q$, which again contradicts Inequality (\ref{t_n approximate inequality}).

But we also have $3p|q^2+q+1$ which forces $3p \leq q^2+q+1$. These two inequalities together form a contradiction.  
\end{proof}

The reader is invited to think about the graph theory interpretation of Lemma \ref{PQR lemma}. 

We also have the following result.
\begin{lemma}\label{No linked sigma22 pairs}  Assume that $p$, $q$ and $r$ are distinct odd primes. Assume further that $p$ and $q$  are a $\sigma_{2,2}$ pair, and that $q$ and $r$ are also a a $\sigma_{2,2}$ pair. Then $\{p,q,r\} = \{3,13,61\}$.
\end{lemma}
\begin{proof} This follows immediately from considering $t_n$ (mod 3).
\end{proof}

In graph terms, Lemma \ref{No linked sigma22 pairs} says that we cannot have three vertices $x$, $y$ and $z$, each of out-degree 2 where $x$ and $y$ both point to each other and $y$ and $z$ both point to each other unless they arise from the triplet $\{3,13,61\}$. 

We will also mention here three questions related  to our results with  $\sigma_{2,2}$ pairs. A general question of interest is how similar results are for other $\sigma_{m,m}$ pairs. We can similarly define quasisolutions for $\sigma_{m,m}$ pairs in an analogous way. In that context, define $t_{m,n}$ via the relationship, $t_{m,1}=t_{m,2}=1$ and for $n>2$, $$t_{m,n+1} = \frac{t_{m,n}^{m+1}-1}{(t_{m,n}-1)t_{m,n-1}}.$$ Note that we have $t_{2,n}= t_n$ in our earlier notation. 

One obvious question in this context then is: if $m+1$ is prime, is it true that all quasisolutions for $\sigma_{m,m}$ pairs arise from $t_{m,n}$? The answer  here is no. In the case when $m=4$, we have that $(1, 1)$, $(5, 11)$, $(61, 131)$, and $(101, 491)$ all produce their own chain of solutions. 

We will note here three open questions.

First, we tentatively suspect the following conjecture. 

\begin{conjecture}\label{square free conjecture} If $p$ and $q$ are a $\sigma_{2,2}$ pair, then $p^2+p+1$ and $q^2+q+1$ are squarefree.

\end{conjecture}
Note that if Conjecture \ref{square free conjecture} is true this would trivially imply Lemma \ref{squared divisor in sigma 2,2 pair}. \\

Second, we also suspect the following statement. 
 Let $L(n)$ be the largest square divisor of $t_n^2 + t_n +1$. Then for any $\epsilon > 0$, we have $L(n) = O(t_n^\epsilon)$. Note that even getting an explicit bound for some reasonably small fixed epsilon would be interesting and useful for tightening the results in this paper. Similarly, let $S(n)$ be the largest square divisor of $((t_n)^2+t_n+1)((t_{n+1})^2+t_{n+1}+1)$. it seems likely that there is a constant $C$ such that for all $n$ we have $S(n) \leq Ct_{n+1}$, and we can likely take $C=1$. 

Third, we have the following question. Are there infinitely many $\sigma_{2,2}$ pairs? We strongly suspect that the answer is no. We have the following heuristic: Inequality (\ref{t_n approximate inequality}) implies that $t_n$ grows at least like $4^n$. The probability that $t_n$ is prime should be bounded above by $\frac{1}{\log 4^n} = \frac{1}{(\log 4) n}$. Thus, the probability that both $t_n$ and $t_{n+1}$ are prime should be bounded above $\frac{C}{n^2}$ for some constant $C$. But $\sum_{n=1}^\infty \frac{C}{n^2}$ is a convergent series, so if we go out far enough, the probability that there are any more such pairs should get very small. 

Finally, in our last remark concerning $\sigma_{2,2}$ pairs, we prove one more minor result. We do  not need this result here, but include it for three reasons. First, this lemma would likely be useful for extending the results in this paper or tightening those results. Second, this lemma can be thought of as a substantial restriction on what the graph of an odd perfect number can look like. Third, this lemma itself is an interesting restriction on what $\sigma_{2,2}$ pairs look like.

\begin{lemma}\label{p2+p+1 and q2 + q+ 1 are almost relatively prime when coming from a sigma 2 2 pair} Suppose that $p$ and $q$ are a $\sigma_{2,2}$ pair. Then either $(p^2+p+1,q^2+q+1)=3(7^m)$ for some non-negative integer $m$, or we have $\{p,q\}=\{3,13\},$ in which case $(p^2+p+1,q^2+q+1)=1$.
\end{lemma}
\begin{proof}  Assume that $p$ and $q$ are a $\sigma_{2,2}$ pair. The case when $\{p,q\}=\{3,13\}$ is a straightforward calculation, so assume without loss of generality that $3 < p < q$. Note that $3|q^2+q+1$ and $3|p^2+p+1$. Now, we'll assume that $k$ is a prime 
such that $k|p^2+p+1$ and $k|q^2+q+1$ and we'll show that $k=3$ or $k=7$. Since $n^2 + n+ 1 \not \equiv 0$ (mod 9) for any $n$, this will suffice to prove the result. 

We have from Equation \ref{quasisolution equation} that $$p^2+p+1= 5pq-q^2-q$$ and hence $$k|(5pq-q^2)=q(5p-q-1).$$ By the same logic we have that $$k|(5pq-p^2)=p(5q-p-1).$$ Since $(k,pq)=1$ we have $k|5q-p-1$ and $k|5p-q-1$. We then have
$$k|(5p-q-1)-(5q-p-1) = 6(p-q).$$ Since $k$ must be odd, we have then $k|3(p-q)$. So either $k=3$ or $k|(p-q)$. For the remainder of this proof, we will assume that $k \neq 3$, and so $k|(p-q)$. We also have $$k|(5p-q-1)+(5q-p-1) = 2(2p+2q-1).$$ Hence, $k|2p+2q-1$, and so $$k|(2p+2q-1) + 2(p-q) = 4p-1.$$
Then, $$k|(p^2+p+1)-(4p-1) = p(p+5).$$ Since $(k,pq)=1$, we have then $k|(p+5)$, and so 

$$k|(p^2 +p +1)- (p+5)^2 +9(p+5) = 21,$$ and hence $k=7$. 
\end{proof}

Note that the above proof can be modified to show that if we have $p$ and $q$ a quasisolution, then $(q^2+q+1,p^2+p+1)|3 (7^m)$ for some non-negative integer $m$.

We also need the following result which concern $\sigma_{4,1}$ pairs. 

\begin{lemma}\label{sigma 41 pair lemma 2} If $p$ and $q$ are odd primes, with $p|q+1$ and $q|\sigma(p^4)$, then we have that $p^2 \not |(q+1)$.
\end{lemma}
\begin{proof} Assume that $p$ and $q$ are odd primes. Assume also that $p^2|q+1$, and $q|\sigma(p^4)$. We can easily check that we must have $$q> p \geq 7.$$ 
We may choose $m$ such that $p^2m=q+1$. We then have
$$q|\left(m\left(p^4 + p^3 + p^2 + p +1\right) - p^2\left(p^2m -1 \right) - p\left(p^2m-1\right) - \left(p^2m-1\right)\right).$$
This is the same as 
$$q|(mp + m + p^2 + p + 1).$$
We have that 
$$q \leq \frac{q+1}{p} + \frac{q+1}{p^2} + \frac{q+1}{m} + \frac{q+1}{mp} + 1 \leq (q+1)\left(\frac{1}{7} + \frac{1}{49} + \frac{1}{2} + \frac{1}{14} \right) + 1 < q.$$ which is a contradiction. 
\end{proof}

Before we continue, we note two arguments we will frequently make which are simple enough that neither do rises to the level of a lemma. However, both  are worth noting explicitly.

First, when we have two odd primes $x$ and $y$ and $x <y$, we must have $y \not|\sigma(x)$, since this would force $y \leq \frac{x+1}{2} < x < y$.

Second, and in similar vein, if we have three odd primes, $x$, $y$, and $z$ with $x < y \leq z$, then we cannot have $yz|\sigma(x^2)$, since $$yz \geq (x+2)^2 = x^2+4x+4 > x^2+x+1 = \sigma(x^2).$$

Finally, note that we will occasionally need the fact that any odd perfect number has at least four distinct prime divisors, and on one occasion we'll use that an odd perfect number must have at least five distinct prime divisors. In that context, we note that the best current result in this direction is Nielsen's result \cite{Nielsen}  that an odd perfect number must have at least ten distinct prime factors.

\section{Bounding $abc$}

Before we prove the main result, we prove an easier bound on $abc$, similar to how we proved Theorem \ref{Generalbound}. The proof of the main result uses a similar method. The main result is substantially easier to follow if one first proves this weaker result, which demonstrates many of the central ideas behind the main theorem. 

\begin{theorem}\label{Easy abc bound} We have $abc < 2^{\frac{5}{12}}3^{\frac{7}{36}}N^{\frac{11}{18}}$.
\end{theorem}

Note that $2^{\frac{5}{12}}3^{\frac{7}{36}} = 1.6527 \cdots$, so a slightly weaker but cleaner version of this statement is that $abc < 2N^{\frac{11}{18}}$.

Before we prove Theorem \ref{Easy abc bound} a few remarks on our tactics. We will have a few easy cases. The harder cases will involve obtaining a series of inequalities which are linear in $\log a$, $\log b$, $\log c$, and $\log N$.  We will then take a linear combination of those inequalities to get the inequality from Theorem \ref{Easy abc bound}. The choices of coefficients for the linear combinations may appear to the reader as haven arisen with no motivation. However, they were obtained by performing linear programming on the dual of the system of linear inequalities. This linear programming then gives optimal linear combinations to prove the best cost inequalities. 
We'll also need to rewrite some of our earlier inequalities as linear combinations in this way. For the remainder of this section we will write $\alpha =\log a$, $\beta = \log b$, and $\gamma = \log c$. We then have the following inequalities. 
 
 Acquaah and Konyagin's Inequality (\ref{AK inequality}) is equivalent to
 \begin{equation} \label{AK log form}  3\gamma \leq \log N + \log 3 .
 \end{equation}
 
 Similarly, Inequality (\ref{bc inequality from previous paper}) is equivalent to
 
 \begin{equation}\label{log form of earlier bc inequality}
2 \beta + 2 \gamma \leq \log N + \frac{1}{2}\log 6 .
\end{equation}

Much of the proof of Theorem \ref{Easy abc bound} will be encapsulated in the following Lemma.
\begin{lemma} If we have $a^3b^2c \leq 2N$ then 
$$abc \leq 2^{\frac{5}{12}}3^{\frac{7}{36}}N^{\frac{11}{18}}.$$ \label{Big chunk of easy abc bound} \end{lemma}
\begin{proof} Assume that we have $a^3b^2c \leq 2N$. Then, using our earlier notation, this is the same as \begin{equation}\label{R2 equation}
3\alpha + 2\beta + \gamma \leq \log N + \log 2.    
\end{equation} We then add our inequalities as follows (with each equation's number in bold). We take   $\frac{1}{9}{\bf{\ref{AK log form}}}  + \frac{1}{6}{\bf{\ref{log form of earlier bc inequality}}}$ + $\frac{1}{3}{\bf{\ref{R2 equation}}}$, which yields 
$$\alpha + \beta + \gamma \leq \frac{11}{18}\log N + \frac{5}{12}\log 2 + \frac{7}{36}\log 3$$ which is equivalent to the desired inequality. 
\end{proof}

We are now ready to prove Theorem \ref{Easy abc bound}.

\begin{proof} If we have $a^2|N$, $b^2|N$ and $c^2|N$, then we have $a^2b^2c^2 < N$ and hence $abc < N^{1/2}$.  We thus may assume that of $a$, $b$ and $c$ one of them is the special prime and is raised to the first power (By Lemma \ref{Euler form for OPN} an odd perfect number has exactly one prime raised to the first power). We will assume that $c$ is the special prime; the cases where $a$ or $b$ is the special prime look nearly identical. If $a$ or $b$ is raised to a power higher than the second, we have either $a^4|N$ or $b^4|N$. 

If we have $a^4|N$, then we have $a^4b^2c|N$ and so  
$$a^3b^2c < a^4b^2c \leq N < 2N.$$ Hence, we may invoke Lemma \ref{Big chunk of easy abc bound}. Similarly, if we have $b^4|N$, then we have
$$a^3b^2c < a^2b^4c \leq N < 2N$$ and we may then again invoke Lemma \ref{Big chunk of easy abc bound}. Thus, we may assume that we have $a^2||N$ and $b^2||N$. Since an odd perfect number must have more than three distinct prime factors, $a^2b^2c$ is a proper divisor of $N$. Because any proper divisor of a perfect number must be deficient, $a^2b^2c$  must be deficient. We may then use an $m$-type argument. In particular, we must have $\sigma(a^2b^2c) < 2\sigma(a^2b^2c)$, and thus there is a prime $p$, where $p \in \{a,b,c\}$, and a component $m_p$ of $N$  such that
$(m_p,abc)=1$, and $p|\sigma(m_p)$. Since $m_p$ is a power of an odd prime we have that $$\frac{p}{m_p} \leq  \frac{\sigma(m_p)}{m_p} < \frac{3}{2},$$ 
and thus $$p < \frac{3}{2}m_p.$$ Since $p \geq a$, we have that $m_p \geq \frac{3}{2}a.$ Since $m_p|N$, and $(m_p, abc)=1,$ we have
$$\left(\frac{3}{2}a\right)a^2b^2c < m_p a^2b^2c \leq N$$
and so 
$$a^3b^2c < \frac{2}{3}N < 2N$$ which allows us to use Lemma \ref{Big chunk of easy abc bound}, completing the proof. 
\end{proof}

We are now in a position to state and prove the main theorem. 

\begin{theorem}\label{abc < N to the three fifths} We have $abc < (2N)^{\frac{3}{5}}.$ 
\end{theorem}

For convenience we will prove Theorem \ref{abc < N to the three fifths} as a series of separate propositions.

We'll note for convenience that we also have the trivial inequalities

\begin{equation}\label{alpha less than beta} \alpha- \beta <0, \end{equation}

and

\begin{equation}\label{beta less than gamma}\beta - \gamma <0.  \end{equation}

Also note that Inequality (\ref{b bound from Z1}) is equivalent to

\begin{equation} \label{b bound from Z1 log form}5\beta \leq \log N + \log 2. 
\end{equation} 

\begin{proposition} If $a^4|N$, $b^4|N$ or $c^4|N$ then we have
$abc <  2^{\frac{7}{20}}3^{\frac{13}{60}}N^{\frac{17}{30}}$.
\end{proposition}
\begin{proof} Assume that at least one of $a^4|N$, $b^4|N$ or $c^4|N$. By the same logic as in the proof of Theorem \ref{Easy abc bound}, we must have
$$a^5b^2c < 2N.$$ We then have
\begin{equation}\label{a5 b2 c log form} 5\alpha + 2\beta + \gamma < \log N + \log 2 . 
\end{equation}
We take 
 $\frac{1}{15}{\bf{\ref{AK log form}}}  + \frac{3}{10}{\bf{\ref{log form of earlier bc inequality}}}$  $+ \frac{1}{5}{\bf{\ref{a5 b2 c log form}}}$, which yields $$\alpha + \beta + \gamma \leq \frac{17}{30}\log N + \frac{13}{60}\log 3 + \frac{7}{20}\log 2, $$ which yields the desired inequality. 

\end{proof}

\begin{proposition} Assume that $a^2|N$, $b^2|N$ and $c^2|N$. Then $abc < N^{1/2}$
\end{proposition}
\begin{proof}  This lemma essentially amounts to just observing that $a^2b^2c^2 < N$ and then taking the square root of both sides. 

\end{proof}

Strictly speaking, we do not need the next result, but it may be of interest to see how far we can push the above.

\begin{proposition} Assume that $a^2||N$, $b^2||N$ and $c^2||N$. Then $$abc < 2^{\frac{1}{3}}3^{\frac{1}{18}}N^{\frac{17}{36}}.$$ 
\end{proposition}
\begin{proof} Assume that $a^2||N$, $b^2||N$ and $c^2||N$. We can use an $m$-type argument to obtain that
$$a^3b^2c^2 < 2N$$ which becomes 
\begin{equation}\label{3alpha + 2beta + 2gamma < log N I} 3\alpha + 2\beta + 2\gamma < \log N + \log 2. 
\end{equation}

Note that $c^2 \not |\sigma(b^2)$.  
Note also that we cannot have $c^2|\sigma(a^2)$ nor can we have $b^2|\sigma(a^2)$ nor $bc|\sigma(a^2)$. 
 If we have $c \not|\sigma(b^2)$, then we have that $$b^4c^2 < (b^2\sigma(b^2)c^2) |N $$ and therefore\begin{equation}\label{4 beta + 2 gamma} 4\beta + 2\gamma \leq \log N. \end{equation}
 
 We then take  $\frac{1}{18}{\bf{\ref{AK log form}}}+ \frac{1}{3}{\bf{\ref{3alpha + 2beta + 2gamma < log N I}}}+\frac{1}{12}{\bf{\ref{4 beta + 2 gamma} }} $ which yields $$\alpha + \beta + \gamma < \frac{17}{36}\log N + \frac{1}{18}\log 3 + \frac{1}{3}\log 2,$$ which is equivalent to the desired inequality.



Thus, we may assume that we are in the situation where $c|\sigma(a^2)$ and $b \not |\sigma(a^2)$. Since we cannot have $c^2 \not|\sigma(a^2)$ we can then use an $m$-component argument to get that
$$a^2b^2c^3 < 2N$$
or equivalently, that 
\begin{equation}\label{2alpha 2beta 3gamma I} 2\alpha + 2\beta + 3\gamma < \log N + \log 2. \end{equation}

We then take as our sum $\frac{1}{7}{\bf{\ref{alpha less than beta}}} + \frac{2}{7}{\bf{\ref{beta less than gamma}}}+  \frac{3}{7}{\bf{\ref{2alpha 2beta 3gamma I}}}$, which yields

$$\alpha + \beta + \gamma \leq \frac{3}{7}\log N + \frac{3}{7}\log 2.$$

This is the same as

$$abc < (2N)^{\frac{3}{7}}, $$ which implies the desired inequality. 

We have completely handled the situation where $c \not|\sigma(b^2)$. We  may now assume that $c|\sigma(b^2)$. Again, note that we must have $c^2 \not|\sigma(b^2)$. Note that if we have $b\not| \sigma(c^2)$, then we have
$$b^3c^3<  (b\sigma(c^2)\sigma(b^2)c)|N,$$ which implies

\begin{equation}\label{3beta 3gamma < log N}  3\beta + 3\gamma \leq \log N. 
\end{equation}

We take then as our sum $\frac{1}{3}{\bf{\ref{3alpha + 2beta + 2gamma < log N I}}} $ + $\frac{1}{9}{\bf{\ref{3beta 3gamma < log N}}}$ which yields

$$\alpha + \beta + \gamma < \frac{4}{9}\log N + \frac{1}{3}\log 2$$ which implies the desired inequality. 

We may thus assume that $b|\sigma(c^2)$. So $b$ and $c$ form a $\sigma_{2,2}$ pair.  By Lemma \ref{squared divisor in sigma 2,2 pair}, we have $b^2\not|\sigma(c^2)$, and so
$$(b\sigma(b^2)c\sigma(c^2))|N$$ 

Note that if $a|\sigma(c^2)$, then, since $b$ and $c$ form a $\sigma_{2,2}$ pair, we cannot have $a$ and $c$ be a $\sigma_{2,2}$ pair since if they were, we'd have $a=3$ by Lemma \ref{No linked sigma22 pairs}. But we must have $a>100$ due to Iannucci's result, so this is impossible.\footnote{An alternate way of reaching a contradiction here is to note that if the third largest prime factor were $3$, then $N$ would only have three distinct prime factors.} Thus, in this case we may assume that $c\not|\sigma(a^2)$. An $m$-type argument gives us again that
$$a^2b^2c^3 < 2N$$ and our logic then goes through as before to obtain the result that $$abc < (2N)^{\frac{3}{7}} .$$ We may thus assume that $a \not|\sigma(c^2)$.

Now, consider what $a$ may divide. If $(a,\sigma(b^2)\sigma(c^2)$)=1 then we have
$$a^2b\sigma(b^2)c\sigma(c^2)|N,$$ which yields that

\begin{equation}\label{2alpha 3beta 3gamma < log N} 2\alpha + 3\beta + 3\gamma < \log N .\end{equation}

We may take as our sum $\frac{1}{8}{\bf{\ref{beta less than gamma}}} $ $\frac{1}{4}{\bf{\ref{alpha less than beta}}} $ $+ \frac{3}{8}{\bf{\ref{2alpha 3beta 3gamma < log N}}}$ to get that

$$\alpha + \beta + \gamma < \frac{17}{36}\log N.$$

We then have that

$$abc  < N^{\frac{17}{36}}.$$

We may thus assume that either $a|\sigma(b^2)$ or $a|\sigma(c^2)$. We will only look at the first case (the second case is nearly identical).  If this is true, then by Lemma \ref{No linked sigma22 pairs}, we have that $b\not| \sigma(a^2)$ and by Lemma \ref{squared divisor in sigma 2,2 pair} that $b^2 \not |\sigma(c^2)$, so we may make an $m$-type argument to obtain that
$$a^2b^3c^2 < 2N,$$ which we have already seen is an inequality strong enough to obtain our result. 
\end{proof}

Note that if we knew Conjecture \ref{square free conjecture}, then the above proposition could  very likely be tightened. 

We are now in a position where the only remaining cases to be considered are one of $a$, $b$ or $c$ is raised to the first power and the other two are raised to the second. 
\begin{proposition} If $a||N$, $b^2||N$ and $c^2||N$, then 
$$abc < N^{\frac{1}{2}}.$$
\end{proposition}
\begin{proof}  Assume that $a||N$, $b^2||N$ and $c^2||N$. Since $\frac{a+1}{2} < a < b < c$, we have that that $b \not|\sigma(a)$ and $c\not|\sigma(a)$. Hence,  $$a^2b^2c^2 < (a\sigma(a)b^2c^2)|N,$$ from which the result follows.
\end{proof}

\begin{proposition} Suppose that $a^2||N$, $b||N$ and $c^2||N$. Then we have $$abc < 2N^{\frac{11}{20}}.$$ 

\end{proposition}
\begin{proof} 


First, note that $c\not|\sigma(b)$, since $c> b > \frac{b+1}{2}$. 

We will first consider the situation where $a^2|\sigma(b)$. In that situation  we have $a^2 < \frac{b+1}{2}<b$, and thus we also have $b\not|\sigma(a^2)$.  Note that we also have $a^2 +a +1 < b < c$ and so we have $c \not|\sigma(a^2)$.  We then have  

$$a^2b^2c^2< (\sigma(a^2)b\sigma(b)c^2)|(2N).$$ We then have 
$$abc < (2N)^{\frac{1}{2}}.$$

We may thus assume that $a^2\not|\sigma(b)$. 

If $a \not|\sigma(b)$, then we have $$a^2b^2c^2 < (a^2b\sigma(b)c^2)|(2N),$$ and hence we get the same bound as before. That is, $$abc < (2N)^{1/2}. $$ We may thus assume that $a||\sigma(b)$.

By Lemma \ref{p|q+1 and q|p^2+p+1 lemma}, we have that $b\not|\sigma(a^2)$. We also have that $c^2 \not|\sigma(a^2)$ (since this would force $c< a$). We then have 
$$(a^2\sigma(a^2)bc)|N.$$ Since $c \not|\sigma(b)$ we also have 
$$(a\sigma(a^2)b\sigma(b)c)|(2N). $$


Suppose that $c\not|\sigma(a^2)$. In that case, we have 
$$(a\sigma(a^2)b\sigma(b)c^2)|(2N),$$ which yields
$$abc < 2N^{\frac{1}{2}} . $$

We may thus assume that $c|\sigma(a^2)$. Now, suppose that $a \not |\sigma(c^2)$. Then we have $$(a\sigma(a^2)\sigma(b)c\sigma(c^2))|(2N). $$ This implies that
$$a^3bc^3 < 2N,$$

and again we have $$abc < 2N^{\frac{1}{2}}.$$

Note that with a little work we can actually tighten this last case slightly from $a^3bc^3 <2N$ to get 

$$abc < 2N^{\frac{7}{15}}$$ but we will not need that here. 

We may now assume that $a|\sigma(c^2)$, and so $a$ and $c$ form a $\sigma_{2,2}$ pair. Then, since the special prime must be $1$ (mod 4), we may invoke Lemma \ref{PQR lemma} to conclude that $b\not |\sigma(c^2)$ since otherwise $c$ and $b$ would form a $\sigma_{2,2}$ pair.  We then obtain
$$(a\sigma(a^2)bc\sigma(c^2))|N,$$ which again
yields that $$a^3bc^3 < 2N$$ and the logic is again identical. 
\end{proof}

We now have our last situation. (Note that the below proposition is the weakest inequality, and so any improvement in the main theorem would come from improving this proposition.) 

\begin{proposition} Suppose that $a^2||N$, $b^2||N$ and $c||N$. Then we have $$abc < (2N)^{\frac{3}{5}}.$$ 
\end{proposition}
\begin{proof} Assume that $a^2||N$, $b^2||N$ and $c||N$. We have, from an $m$-type argument, that $$a^3b^2c < 2N,$$
 
 which becomes
 
 \begin{equation}\label{3alpha + 2beta + gamma < log N + log 2} 3\alpha + 2\beta + 3\gamma <  \log N + \log 2.
 \end{equation}

Note that if $(ab, \sigma(c))=1$, then we  have $$(a^2b^2c\sigma(c))|(2N),$$ in which case we immediately have 
$$a^2b^2c^2 < 2N $$ and hence
$$abc < (2N)^{\frac{1}{2}} < (2N)^{\frac{3}{5}}.$$
We may thus assume that either $a|\sigma(c)$ or $b|\sigma(c)$. 

Now, assume that $(ac, \sigma(b^2))=1$. In that case
we have $$a^2b^4c < (a^2b^2\sigma(b^2)c) | N.$$ 

We get then 

\begin{equation}\label{2alpha + 4beta + c} 2\alpha + 4\beta + c \leq \log N . \end{equation}

We take as our sum $\frac{2}{9}{\bf{\ref{AK log form}}} +$ $\frac{1}{3}{\bf{\ref{alpha less than beta}}} $ $+ \frac{1}{3}{\bf{\ref{2alpha + 4beta + c}}}$ which yields 

$$\alpha +\beta + \gamma < \frac{5}{9}\log N + \frac{2}{9}\log 3 .  $$

We immediately obtain 
$$abc <  3^{\frac{2}{9}}N^{\frac{5}{9}} < 2N^{\frac{3}{5}} .$$

We  may thus assume that we have $a|\sigma(b^2)$ or $c|\sigma(b^2)$

Let us consider the case where $ac|\sigma(b^2)$.  Then we have $ac \leq b^2 + b + 1 < 2b^2$. 

We thus have
\begin{equation}\label{alpha + gamma < 2beta} \alpha + \gamma - 2\beta <   \log 2.
\end{equation}

We may then take as our sum $\frac{3}{5}{\bf{\ref{b bound from Z1 log form}}} $  $+ {\bf{\ref{alpha + gamma < 2beta}}} $ which again yields

$$\alpha + \beta + \gamma < \frac{3}{5}\log N + \frac{3}{5}\log 2. $$

We may thus assume that we do not have both $a|\sigma(b^2)$ and $c|\sigma(b^2)$.  Let us first consider the case where $c|\sigma(b^2)$ and $a\not|\sigma(b^2)$. From Lemma \ref{p|q+1 and q|p^2+p+1 lemma} we have $b \not|\sigma(c)$. Note that we also have $b^2 \not|\sigma(a^2)$ and so we have that
$$a^2b^3c < (\sigma(a^2)b\sigma(b^2)\sigma(c))|(2N)$$ which 
we have seen is enough to obtain that $$abc < 2^{\frac{3}{5}}N^{\frac{3}{5}}.$$

Now, let us consider the case where $a|\sigma(b^2)$, and $c\not| \sigma(b^2)$. Assume for now that $a^2|\sigma(b^2)$. Then, by Lemma \ref{squared divisor in sigma 2,2 pair}, we have $b\not|\sigma(a^2)$. Now, if $c \not| \sigma(a^2)$, then we have
$$(a^2\sigma(a^2)b^2c) |N,$$ which yields $$abc < 2N^{\frac{7}{12}}.$$ So we may assume that $c|\sigma(a^2)$. We then have that $b^2 \not|\sigma(c)$, since it would force $b<a$. If $b\not|\sigma(c)$, then we would have 
$(a^2b^2c\sigma(c))|(2N)$ which yields $$abc < (2N)^{\frac{1}{2}}.$$  We may thus assume in this context that $b||\sigma(c)$. By an $m$-type argument we then have 
$$a^2\frac{1}{2}b^3c \leq N,$$ which again yields that $abc < N^{\frac{3}{5}}.$ We may thus assume that $a||\sigma(b^2)$. Then we have $$(ab^2\sigma(b^2)c)|(2N),$$ which again implies $$a^2b^3c < 2N$$ and so we are done with this case. 

Now, if $c|\sigma(b^2)$, then we also have that $b \not|\sigma(c)$ by Lemma \ref{p|q+1 and q|p^2+p+1 lemma}. We then have 
$$ac <\sigma(b^2) <  2b^2$$

and also 

$$b^2c\sigma(c) < N.$$

This last pair of inequalities is again strong enough to get our desired bound. 
\end{proof}

\section{Towards an improvement of bounds on $a$}

One would like to get a bound on $a$ of the form $a< CN^{\epsilon}$ for some $\epsilon < \frac{1}{6}$. This seems difficult. In this section, we will show that one can do so as long as one is not in the situation $a^2||N$, $b^2||N$, and $c||N$. As before, we will break the cases we care about into a variety of different propositions. 

\begin{proposition} If $p^4|N$ for some prime $p \in \{a,b,c\}$, then we have $a < N^\frac{1}{7}$. 
\end{proposition}
\begin{proof} Assume that $p^4|N$ for some prime $p \in \{a,b,c\}$. Then we must have $a^7< a^4b^2c|N$, from which the result follows. 
\end{proof}

We may thus assume going forward that we have $a$, $b$, and $c$ raised to at most the second power.

\begin{proposition} Assume that $a^2||N$, $b^2||N$, and $c^2||N$. Then $a < (2N)^{\frac{1}{7}}$.
\end{proposition}
\begin{proof} Under these assumptions, we have by an $m$-type argument that $a^3b^2c^2 < 2N$. Since $a^7< a^3b^2c$, the result follows.
\end{proof}

\begin{proposition} If $a||N$, $b^2||N$, and $c^2||N$, then $a < (2N)^{\frac{1}{7}}$.
\end{proposition}
\begin{proof}  Assume that $a||N$, $b^2||N$, and $c^2||N$. Note that $(bc, \sigma(a))=1$, since $\frac{a+1}{2} < b < c$. 


If $b\not|\sigma(c^2)$, then  $$a^7 < (\sigma(a)b^2c^2 \sigma(c^2))|(2N).$$ Thus, we may assume that $b|\sigma(c^2)$.

If $a\not|\sigma(c^2)$, and $b||\sigma(c^2)$, then
$$a^7 < a\sigma(a)bc^2\sigma(c^2).$$

So we may assume that either $b^2|\sigma(c^2)$ or $a|\sigma(c^2)$.  If $ab^2|\sigma(c^2)$, then we have
\begin{equation}\label{a^3 < 2c^2}  a^3 < ab^2 < 2c^2. \end{equation}

If $c\not |\sigma(b^2)$, then $$a^7 < (\sigma(a)\sigma(b^2)b^2c^2) | (2N), $$ so we may assume that  $c|\sigma(b^2)$. Since $c|\sigma(b^2) $
by Lemma \ref{squared divisor in sigma 2,2 pair}, we must have $b^2 \not|\sigma(c^2)$, and so we have $b||\sigma(c^2)$, and thus may assume that $a|\sigma(c^2)$.  Since $a||N$, and $a|\sigma(c^2)$, we must then have $a \not|\sigma(b^2)$. (We could also reach this conclusion via Lemma \ref{p2+p+1 and q2 + q+ 1 are almost relatively prime when coming from a sigma 2 2 pair}.)

We then have 

$$a^7 < a\sigma(a)\sigma(b^2)b^2c < 2N$$

and so we are done.

\end{proof}

\begin{proposition} If $a^2||N$, $b||N$ and $c^2||N$, then $a < (2N)^{\frac{1}{7}}.$
\end{proposition}
\begin{proof} Assume that $a^2||N$, $b||N$ and $c^2||N$. If we have $a^2|\sigma(b)$, then since we have $b \leq (2N)^{\frac{1}{5}}$, we have 
$$a < \sqrt{\frac{b+1}{2}} < b < (2N)^{\frac{1}{10}} < (2N)^{\frac{1}{7}}. $$ Thus, we may assume that $a^2 \not|\sigma(b)$.

Note that $c^2 \not |\sigma(a^2)$, and $c\not |\sigma(b)$. We also have that $\sigma(a^2) < c^2 $ and so $c^2 \not|\sigma(a^2)$. We claim that we also must have $bc \not|\sigma(a^2)$. To see this, note that $$bc > (a+2)(a+4)= a^2 + 6a + 8 > a^2+a+1=\sigma(a^2). $$

Note that if $(bc, \sigma(a^2))=1$, then we have
$$a^7 < a^2\sigma(a^2)bc^2 < 2N.$$ We may thus assume that we have exactly one of $b|\sigma(a^2)$ and $c|\sigma(a^2)$.

First, let us consider the case when $b|\sigma(a^2)$ and $c\not|\sigma(a^2)$. We may apply Lemma \ref{p|q+1 and q|p^2+p+1 lemma} to conclude that $a\not|\sigma(b)$. We then have 
$$a^7 < a\sigma(a^2)b\sigma(b)c^2 < 2N$$ which implies the desired bound.

Now, consider the possibility that $c|\sigma(a^2)$ and $b\not|\sigma(a^2). $ We already established that $a^2\not|\sigma(b)$, and so we have 
$$a^7 <  (a\sigma(a^2)b^2c\sigma(b))|(2N)$$ which again gives us our desired bound. 
\end{proof}

Putting all the above propositions from this section together, we have the following dichotomy. 

\begin{theorem} Either $a < (2N)^{\frac{1}{7}}$ or we have $a^2||N$, $b^2||N$ and $c||N$.
\end{theorem}

One obvious question is what we can say about this last situation. In that regard we have the following result. 

\begin{proposition}If $a^2||N$, $b^2||N$ and $c||N$, then either $a < (2N)^\frac{1}{7}$, or all the following must hold: We have $c|\sigma(a^2)$, $b|\sigma(c)$, and $a^2|\sigma(b^2)$. There exists a prime $d$ and a positive integer $j$ such that 
\begin{enumerate}
    \item $d \not \in \{a,b,c\}$
    \item $d^j||N$   
    \item $b|\sigma(d^j)$
    \item $d|\sigma(a^2)$
    \item $d^j \not|\sigma(a^2b^2c)$. 
    \item $d^j < \frac{1}{2}a^2$.
\end{enumerate}
\end{proposition}
\begin{proof} We will assume that we have $a^2||N$, $b^2||N$ and $c||N$, and that the first case above does not hold. Note that we may assume that
$(bc, \sigma(a^2))>1$ since if $bc$ and $\sigma(a^2)$ are relatively prime, we would have
$$a^7 < (a^2b^2c\sigma(a^2))|(2N).$$ As before, we cannot have $bc|\sigma(a^2)$ so we have exactly one of $b|\sigma(a^2)$ or $c|\sigma(a^2)$. 

Let us first consider the case where $b|\sigma(a^2)$ and $c\not|\sigma(a^2)$. Note that if $a\not|\sigma(b^2)$ then $$a^7 < (a^2b\sigma(b^2)c\sigma(a^2)) |(2N).$$ 
Therefore, we may assume that $a|\sigma(b^2)$. Since $a$ and $b$ form a $\sigma_{2,2}$ pair, we have by Lemma \ref{squared divisor in sigma 2,2 pair} that $a^2\not|\sigma(b^2)$.

Now, if we have $(ab,\sigma(c))=1$, then we have 
$$a^7 < a\sigma(b^2)\sigma(a^2)b\sigma(c)|2N$$ so we may assume that either $a|\sigma(c)$ or $b|\sigma(c)$.
Let us first consider the case where $b|\sigma(c)$. We must have, by Lemma \ref{p|q+1 and q|p^2+p+1 lemma}, that $c\not|\sigma(b^2)$, and hence 
$$a^7 < (abc\sigma(b^2)\sigma(a^2))|N$$  We may assume that $b \not|\sigma(c)$, and hence that $a|\sigma(c)$. By Lemma \ref{PQR lemma}, and again using that the special prime must be $1$ (mod 4), we must have that $c\not|\sigma(b^2)$. So again we obtain $$a^7< ab\sigma(a^2)\sigma(b^2)c|N.$$

We now consider the case where $c|\sigma(a^2)$, and $b\not|\sigma(a^2)$. By Lemma \ref{p|q+1 and q|p^2+p+1 lemma}, we have $a \not|\sigma(c)$. Now, if $b\not|\sigma(c)$, we then have that $$a^7 < (a^2\sigma(a^2)b^2\sigma(c))|(2N), $$
so we may assume that $b|\sigma(c)$. Now, note that if 
$a^2 \not|\sigma(b^2)$, then we have
$$a^7 < a\sigma(a^2)\sigma(b^2)b^2 < N,$$ and so we have $a^2 |\sigma(b^2)$.

We have already established that $b|\sigma(c)$. We now wish to show that $b||\sigma(c)$. Assume that $b^2|\sigma(c)$, then we have
$$\sigma(a^2b^2c) =\sigma(a^2)\sigma(b^2)\sigma(c) \geq ca^2 2b^2 = 2a^2b^2c.$$ But that would mean that $a^2b^2c$ is either perfect or abundant and is a proper divisor of $N$, which is a contradiction. Hence the assumption that $b^2|\sigma(c)$ must be false.

By an $m$-type argument, there is a prime $d$ and and a positive integer $j$ such that $d^j||N$, $d \not \in \{a,b,c\}$, and $b|\sigma(d^j)$. 

Since $b|\sigma(d^j)$ we have that $d^j> \frac{1}{2}b$. Now, if $d \not|\sigma(a^2)$, then we have 
$$ \frac{1}{2}a^7 < a^4b^2\frac{1}{2}b< (a^2 \sigma(a^2)b^2d^j) |N.$$

So we may assume that $d|\sigma(a^2)$. Now, assume that $d^j|\sigma(a^2b^2c)$. In that case we
have $(a^2b^2cd^j)|\sigma(a^2b^2cd^j)$ so $a^2b^2cd^j$ is perfect or abundant, which is impossible since $a^2b^2cd^j$ is a proper divisor of $N$.  (Note that here we are using that an odd perfect number must have at least five distinct prime factors.) 

We now just need to prove Item 6. So assume that $d^j \geq \frac{1}{2}a^2.$ Then we have
$$a^5< a^2b^2c <  \frac{N}{d^j} < \frac{2N}{a^2},$$ and we can then solve the resulting inequality for $a$. 
\end{proof}

Note that we can improve Item 6's bound by using the fact that an odd perfect number must be divisible by more primes, and so we can replace the $\frac{1}{2}$ in Item 6 with a much smaller constant.

\section{Towards an improvement of bounds on $bc$}

The situation for trying to improve the bound on $bc$ is very similar to that with $a$. Namely, we can get tighter bounds in all cases except for certain specific contexts when $b^2||N$ and $c||N$.

\begin{proposition}
If $N$ is an odd perfect number, with $b^2||N$, and $c^2||N$, then $$bc \leq 2(3^{1/3})N^{\frac{5}{12}}.$$
\end{proposition}
\begin{proof} Assume $b^2||N$, and $c^2||N$.  If we have that $c \not|\sigma(b^2)$ and $b \not|\sigma(c^2)$ then we have that $$b^4c^4 < b^2\sigma(b^2)\sigma(c^2)|2N,$$ and so 
$bc < 2N^{\frac{1}{4}}.$ We thus may assume that either $b|\sigma(c^2)$ or that $c|\sigma(b^2)$. Note that $c^2 \not|\sigma(b^2)$. To see why, note that $b^2+b+1$ is not a perfect square; so if $c^2|(b^2+b+1)$ we  must have $3c^2 \leq b^2+b+1$. But that would force $c<b$.

Now, assume that $c \not |\sigma(b^2)$. Then we have $$b^4c^2 < (b^2\sigma(b^2)c^2) |(2N),$$
and so $b^4c^2 < 2N$. Now set $c=N^\alpha$. Then $$b \leq \left(\frac{2N}{N^{2\alpha}}\right)^{\frac{1}{4}} < 2N^{\frac{1}{4}-\frac{\alpha}{2}}.$$ Then
$$bc < 2N^{\frac{1}{4}-\frac{\alpha}{2}}N^{\alpha} = 2N^{\frac{1}{4} + \frac{\alpha}{2} }$$
We can make this quantity as large as possible by making $\alpha$ as large as possible, which would occur when we have $c= 3^{1/3}N^{1/3}$. Thus,  
$$bc \leq 2(3^{1/6})N^{\frac{5}{12}}.$$

We may thus assume that $c||\sigma(b^2)$. Then by Lemma \ref{squared divisor in sigma 2,2 pair} we have that $b^2 \not|\sigma(c^2)$. We then have that 
$$b^3c^3 < (b\sigma(b^2)c\sigma(c^2)) |(2N)$$ and so $bc \leq (2N)^{1/3}$.
\end{proof}

\begin{proposition}If $b||N$ and $c^2||N$ then $bc \leq (2N)^{2/5}$.
\end{proposition}
\begin{proof} Assume as given. Note that $c \not|\sigma(b)$ since if it did, we would have $c \leq \frac{b+1}{2} < b$. Thus, there exists $m$ such that $m|N$, $(m,N/m)=1$, $(m,bc)=1$, and $c^2|\sigma(m)$. Note that since $N$ is perfect, $m$ is deficient, and so we must have $m > \frac{c^2}{2}$. We then have
$$\frac{1}{2}c^2bc^2 \leq mbc^2 |N$$ and so 
$$b^{\frac{5}{2}}c^{\frac{5}{2}} \leq 2N,$$ from which the result follows.
\end{proof}

\begin{proposition} If either $b^4|N$ or $c^4|N$, then we have that $$bc \leq 4N^{\frac{4}{9}}.$$
\end{proposition}
\begin{proof} First note that if $(b^4c^4)|N$ then $bc < N^{\frac{1}{4}}$ so we only need to handle two cases, $b^4|N$ and $c^4|N$. We may assume that not both are true. We will first consider the case when $c^4|N$. We have two subcases: $b||N$
and $b^2||N$. If $b||N$, then we have that $c \not|\sigma(b)$ and thus 
$$b^3 c^3 < (b\sigma(b)c^4)|(2N).$$ This yields that $bc < (2N)^{\frac{1}{3}}$.
If $b^2||N$, then we have that $$b^3c^3< (b^2c^4)|N$$ and the same inequality results.

We then have  two remaining cases. In the first case, Case I, $b^4|N$, $c||N$. In the second case, Case II, we have $b^4|N$ and $c^2||N$.

We'll handle Case I first. We have either $b^4||N$ or we have $b^6|N$ (we cannot have $b^5||N$ since $c$ is the special prime in this case). If $b^6|N$, then we may set $c=N^\alpha$ for some $\alpha$. Thus we have $$b^6 \leq N^{1-\alpha}$$ and hence $$b \leq N^{\frac{1}{6}- \frac{\alpha}{6}}.$$ We then have
$$bc \leq N^{\frac{1}{6}- \frac{\alpha}{6}}N^{\alpha} = N^{\frac{1}{6}- \frac{5\alpha}{6}}. $$ This last quantity on the right is maximized when $\alpha$ is as large as possible, namely when $N^\alpha = (3N)^{1/3}.$ This yields
with a little work $bc \leq 2N^\frac{4}{9}$. 
Now, consider the scenario of $b^4||N$ and $c||N$.  If $b \not| c+1,$ then we have that
$$b^4c^2 < (b^4c \sigma(c+1)) | (2N).$$ And one gets from the above inequality that 
$$bc \leq 2N^{\frac{5}{12}} < 2N^{\frac{4}{9}}.$$
We may thus assume that $b|c+1$.  We may handle the case when $c \not |\sigma(b^4)$ similarly. We thus have that $b|\sigma(c)$ and $c|\sigma(b^4)$.

We then have by Lemma \ref{sigma 41 pair lemma 2} that $b^2 \not|\sigma(c)$. We then have  that
$$b^3c^2 \leq (b^3c\sigma(c))|2N.$$ Then by similar logic, by setting $c=N^\alpha$ and using this to maximize $bc$ we obtain that $bc < 4N^{\frac{4}{9}}.$

We now consider Case II, where $b^4|N$ and $c^2||N$. This case is enough to get from $b^4c^2 < N$ the desired inequality through the same method as before.

\end{proof}

We now come to the pesky case that is the primary barrier to improvement,  namely $b^2||N$ and $c||N$.

Let's discuss what results we do have in this case.
Using the same techniques as before we easily get the following result.  \begin{proposition} If $c \not|\sigma(b^2)$ and $b \not |\sigma(c)$, then we have that $$bc < 4N^{\frac{5}{12}}.$$
\end{proposition}

Summarizing the above we have the following theorem. 

\begin{theorem} We have either $$bc < 4N^\frac{4}{9}$$ or we must have:
\begin{enumerate}
    \item Both $b^2||N$ and $c||N$
    \item Either $c|\sigma(b^2)$ or $b|\sigma(c)$.
\end{enumerate}

\end{theorem}

We now consider the situations where we have either $b|\sigma(c)$ or $c|\sigma(b^2)$. Note that we cannot have both by Lemma \ref{p|q+1 and q|p^2+p+1 lemma}.  In this context we can prove that we are in a highly restricted situation.

\begin{proposition} Assume that $b^2||N$ and that $c||N$. If $b \not|\sigma(c)$, and $c|\sigma(b^2)$, then there exists an $m$ such that \begin{enumerate}
    \item $m|N$,
    \item $m$ has at most two distinct prime factors,
    \item $(N/m)=1$,  
    \item $(bc,m)=1$
    \item $b^2 |\sigma(m)$.
    \item $m \not|\sigma(c)\sigma(b^2)$

\end{enumerate}
\end{proposition}

\begin{proof} Let $m_0$ be the minimum $m_0$ such that $m_0|N$, $(N/m_0)=1$, $(bc,m_0)=1$, and $b^2|\sigma(m_0)$. Note that $m_0$ must have at most two distinct prime factors since there can be at most two components of $N$ which contribute a $b$ to $\sigma(N)$.  So what remains is to prove Item 6. Assume that $m_0|\sigma(c)\sigma(b^2)$. Then
$$\sigma(m_0b^2c) = \sigma(m)\sigma(b^2)\sigma(c)\geq 2mb^2c.$$ Thus, $mb^2c$ is either abundant or perfect. But $mb^2c$ has at most four distinct prime factors, so we cannot have $mb^2c = N$. Thus $N$ has a perfect or abundant divisor and must itself then be abundant and hence not perfect.
\end{proof}

\begin{proposition} Let $N$ be an odd perfect number with $b^2||N$ and $c||N$, $b\not|\sigma(c)$, and let $m$ be as in the above proposition. Then either $bc < 4N^{\frac{5}{12}}$ or $(m,\sigma(c))>1$ 

\end{proposition}
\begin{proof} Assume that $(m,\sigma(c))=1$. Then we have that $$\frac{1}{2}b^4c^2 < mb^2c\sigma(c)||2N.$$
One thus has $$b^4c^2 < 4N,$$ from which the bound follows. 

\end{proof}

We would like to get same but with $(m,\sigma(b^2))=1$.
If we assume that $(m,\sigma(b^2))=1$ then we have that $$\frac{1}{2}b^2b^2\sigma(b^2)< mb^2\sigma(b^2) |N$$ and this only gives $b< N^{1/6}$ which is not strong enough to improve these results further without some sort of tighter bound on $c$.

\section{Further results on $\sigma_{a,b}$ pairs}

This section contains additional results concerning $\sigma_{a,b}$ pairs. These results are not directly relevant to odd perfect numbers but are independently interesting.

\begin{lemma} Suppose $p$ and $q$ are positive integers with $p|q+1$, and $q|p+1$. Then one must have $(p,q) \in \{(1,1), (1,2), (2,1), (2,3), (3,2)\}$
\label{p|q+1 and q|p+1 lemma}
\end{lemma}
\begin{proof}  Assume that $q|p+1$ and $p|q+1$.  We have $kq= p+1$ for some $k$, and so $p=kq-1$. We then have  that $kq-1 | q+1$, and hence $kq - 1 \leq q+1$. Solving for $k$, we obtain that $$k \leq 1+ \frac{2}{q}.$$ The last inequality implies $k \leq 3$. We will consider three cases $k=1$, $k=2$ or $k=3$.

If $k=1$, then we have $$q-1|q+1,$$ and hence $q-1|2q$. Since $(q-1,q)=1$, this forces $q-1|2$, and  therefore either $q=2$ or $q=3$. These correspond to $p=1$ or to $p=2$, leading to the pairs $(p,q)=(1,2)$, and $(p,q) = (2,3)$

If $k=2$, then $2q-1|q+1$. This implies that there is some $m$ such that $m(2q-1)=q+1$. Notice that if $m \geq 3$ this leads to a contradiction, so we must have $m=1$ or $m=2$. If $m=1$, we have $2q-1=q+1$, and so $q=2$, and thus $p=3$ Thus, the only solution for $m=1$ is $(p,q)=(3,2).$ 

If $m=2$, then we have $2(2q-1)= q+1$ which yields $q=1$ and $p=1$ and thus the solution $(p,q) = (1,1)$.

Finally, we have the possibility that $k=3$, which yields $3q-1|q+1$. We then have  $$m(3q-1) = q+1 $$ for some $m$. If $m \geq 2$ we get a contradiction. Thus we may assume that $m=1$. This gives us $3q-1 = q+1$ which yields $q=1$, and $p=2$, which gives our final point $(p,q) = (2,1).$ 

\end{proof}

From Lemma \ref{p|q+1 and q|p+1 lemma}  we may classify all $\sigma_{1,1}$ pairs. 

\begin{proposition} The only $\sigma_{1,1}$ pairs are $(2,3)$ and $(3,2)$.
\end{proposition}

We will now use this result to better understand $\sigma_{2,3}$ and $\sigma_{3,3}$ pairs.

\begin{lemma} Assume that $(p,q)$ is a $\sigma_{3,3}$ pair. Then we must be in one of four circumstances:
\label{sigma 3 3 breakdown into 4 cases}
\begin{enumerate}
    \item $(p,q)$ is a $\sigma_{1,1}$ pair.
    \item We have $p|(q^2+1)$ and $q|(p^2+1)$. 
    \item We have $p|(q+1)$ and $q|(p^2+1)$.
    \item We have $p|(q^2+1)$ and $p|(q+1)$. 
\end{enumerate}

\end{lemma}

\begin{proof} Assume that $(p,q)$ is a $\sigma_{3,3}$ pair. We must then have $p|q^3+q^2+q+1$ and $q|p^3+p^2+p+1$. Note that we have the factorization
$$x^3 + x^2 + x +1 = (x+1)(x^2+1).$$ Since $p$ and $q$ are primes, and we have $p|(q+1)(q^2+1)$, and $q|(p+1)(p^2+1)$ the result follows. 
\end{proof}

Note that Cases 3 and 4 of Lemma \ref{sigma 3 3 breakdown into 4 cases} are symmetric, so to understand the remaining $\sigma_{3,3}$ pairs we need only concentrate on Cases 2 and 3. We will classify explicitly all solutions for Case 3, and will obtain a restriction on Case 2 very similar to the what we did with $\sigma_{2,2}$ pairs. 

Define the sequence $s_n$ as follows: $s_0=s_1=1$, and for all $n \geq 0$ we set $$s_{n+2} = \frac{s_{n+1}^2 +1}{s_n}.$$

\begin{lemma} Suppose that $x$ and $y$ are positive integers such that $x|y^2+1$ and $y|x^2+1$. Then $(x,y)$ is a pair of consecutive terms in the sequence $s_n$.
\label{p|q^2 +1 and q|p^2+1 classification}
\end{lemma}
\begin{proof} It is immediate that the sequence of $s_n$ consist of integers and are solutions to the equation in question. We need to show that every solution arises from this sequence.

Our proof is very similar to what we did to classify quasichain solutions for $\sigma_{2,2}$ pairs. Note that any pair $x,y$ satisfying $x|y^2+1$ and $y|x^2+1$ must either have $y \neq x$, or must be the pair $(x,y)=(1,1)$.  Set  $z=(x^2+1)/y$. We claim that $z$ and $x$ satisfy the pair of relationships $z|x^2+1$ and $x|z^2+1$. The definition of $z$ immediately implies $z|x^2+1$. The second relationship requires some slight work. We have
$$z^2 +1 =\left(\frac{x^2+1}{y}\right)^2+1 = \frac{x^2 + 4x +(y^2+1)}{y^2}.$$
Note that $x|(x^2+2x)$ and  $x|(y^2+1)$ so we have that $x|x^2 + 2x +(y^2+1)$. Since $(x,y)=1$, we then havethat $$x|\left(\frac{x^2 + 2x +(y^2+1)}{y^2}\right)$$ which is the claimed relationship. Thus, if $x \neq y$, we can construct a smaller pair, $z$, and $y$ which satisfy the same relationship. Thus, all solutions must arise from the initial pair (1,1).
\end{proof}

Note that an easy  induction argument shows that for $n>1$, $s_n = F_{2n-1}$ where $F_n$ is the $n$th Fibonacci number. We strongly suspect that there are only finitely many $n$ such that both $F_{2n-1}$ and $F_{2n+1}$ are prime. Note that since $F_p$ can only be prime when $p$ is prime, the existence of infinitely many pairs of primes $F_{2n-1}$ and $F_{2n+1}$ would correspond to a much stronger version of the twin prime conjecture. However, a heuristic argument similar to the argument that we expect only finitely many $\sigma_{2,2}$ pairs suggests we only have finitely many of these pairs also. \\

Define the sequence $u_n$ as follows: 
We set $u_0=u_1=1$ and apply the following two rules:
$$u_{2k+2} = \frac{u_{2k+1}^2 +1}{u_{2k}}$$
and $$u_{2k+3} =  \frac{u_{2k+2} +1}{u_{2k+1}} .$$

Notice that this sequence is periodic and takes the form 
$$1,1,2,3,5,2,1,1,2,3,5\cdots  $$

\begin{lemma} If $a$ and $b$ are positive integers
satisfying $b|a^2 +1$ and $a|b+1$ then they must arise from a pair of terms from the $u_n$ sequence.
\end{lemma}
\begin{proof} The method of proof is similar to our earlier reductions. Assume that we have a pair $(a,b)$ satisfying $b|(a^2 +1)$ and $a|(b+1)$ which is not a pair of consecutive terms of $u_n$. We may pick a pair which has smallest possible value of $a+b$. We may assume that this pair satisfies $a>5$, $b>5$ and $a \neq b$. If $a>b$, then the pair $\frac{b+1}{a}, b)$ also satisfies the desired divisibility relations but has a smaller sum, that is $\frac{b+1}{a} + b < a+b$, which is a contradiction. Similarly, if $b<a$, then the pair $(a,\frac{a^2+1}{b})$ satisfies the divisibility relations while $a + \frac{a^2+1}{b} < a+b$ which again is a contradiction. 
\end{proof}

\section{Acknowledgements}
Rajdip Palit pointed out that an earlier version of Lemma \ref{p2+p+1 and q2 + q+ 1 are almost relatively prime when coming from a sigma 2 2 pair} was incorrect.


\begin{thebibliography}{99}

\bibitem{AK}P. Acquaah and S. Konyagin,  On prime factors of odd perfect numbers, {\it International Journal of Number Theory}  {\bf{8}} 6 (2012), 1537--1540. 
\bibitem{CaiShenJia} T. Cai, Z. Shen, L. Jia, A congruence involving harmonic sums modulo $p^\alpha q^\beta$, {\it International Journal of Number Theory} {\bf 13} 5 (2017), 1083--1094. 
\bibitem{Dandapat} G. G. Dandapat, J.L. Hunsucker, Carl Pomerance, Some new results on odd perfect numbers, {\it Pacific Journal of Mathematics} {\bf 57} 2 (1975), 359–-364
\bibitem{Ellia} P. Ellia, A remark on the radical of odd perfect numbers, {\it Fibonacci Quarterly} {\bf 50} 3 (2012) 231--234
\bibitem{Iannuccithird} D. Iannucci, The third largest prime divisor of an odd perfect number exceeds one hundred, {\it Mathematics of Computation}    {\bf 69} 230 (2000), 867--879. 
\bibitem{Klurman}O. Klurman, 
Radical of perfect numbers and perfect numbers among polynomial values, {\it  International Journal of Number Theory}  {\bf 12} 3 (2016), 585--591.
\bibitem{Mills}   W. H. Mills, A system of quadratic Diophantine equations, {\it Pacific Journal of Mathematics}  {\bf 3} 1 (1953), 209-220.
\bibitem{Nielsen} P. Nielsen, Odd perfect numbers, Diophantine equations, and upper bounds, {\it Mathematics of Computation} {\bf 84} 295 (2015), 2549--2567
\bibitem{OchemRaoRadical} P. Ochem, M. Rao, Another remark on the radical of an odd perfect number, {\it 
Fibonacci Quarterly} {\bf 52} 3 (2014) 215--217. 
\bibitem{LucaPomerance} F. Luca and C. Pomerance, On the radical of a perfect number, {\it New York Journal of Mathematics} {\bf 16} (2010)  23--30.
\bibitem{Zelinskysecondlargest} J. Zelinsky, Upper bounds on the second largest prime factor of an odd perfect number. {\it International Journal of Number Theory} {\bf 15} 6 (2019).

\end{thebibliography}
\end{document}